\documentclass[11pt,a4paper]{amsart}

\usepackage[english]{babel}
\usepackage{amsmath,amsthm,amssymb,amsfonts}
\usepackage[utf8]{inputenc}
\usepackage{color}

\usepackage{hyperref}
\hypersetup{
  colorlinks   = true, 
  urlcolor     = blue, 
  linkcolor    = blue, 
  citecolor    = red   
}
\usepackage{anysize}

\usepackage{graphics}
\usepackage{graphicx}
\usepackage{amsmath}
\usepackage{amsthm}
\usepackage{amssymb}
\usepackage{graphicx}
\usepackage{multirow}
\usepackage{vmargin}
\usepackage{multicol}
\usepackage{color,soul}
\usepackage{atbegshi}

\newcommand{\R}{\mathbb R} 

\renewcommand{\P}{\mathbb P}

\renewcommand{\epsilon}{\varepsilon}

\newcommand{\conv}{\mathop\mathrm{conv}\nolimits}

\newcommand{\rank}{\mathop\mathrm{rank}\nolimits}
\newcommand{\poly}{\mathop\mathrm{Poly}\nolimits}
\newcommand{\supp}{\mathop\mathrm{supp}\nolimits}

\newcommand*{\where}{\ \ifnum\currentgrouptype=16 \middle\fi|\ }

\newtheorem{theorem}{Theorem}
\newtheorem{proposition}[theorem]{Proposition}

\newtheorem{corollary}[theorem]{Corollary}
\newtheorem{question}{Question}
\theoremstyle{definition}
\newtheorem{definition}[theorem]{Definition}
\theoremstyle{remark}

\numberwithin{equation}{section}

\title{Weak $(1-\epsilon)$-nets for polynomial superlevel sets}

\author{Pablo Gonz\'alez-Maz\'on}
\address{Universit\'e C\^ote d'Azur, Inria, 2004 route des Lucioles, 06902 Sophia Antipolis, France}
\email{pablo.gonzalez-mazon@inria.fr}

\author{Alfredo Hubard}
\address{Universit\'e Gustave Eiffel, LIGM, Marne-la-Vall\'ee, France}
\email{alfredo.hubard@univ-eiffel.fr}

\author{Roman Karasev}
\address{ Institute for Information Transmission Problems RAS, Bolshoy Karetny per. 19, Moscow, Russia 127994}
\email{r\_n\_karasev@mail.ru}

\thanks{Part of this work appears in the master's thesis of PGM under the direction of AH at the Universit\'e Gustave Eiffel. During the master's PGM was funded by the scholarship of the Labex B\'ezout. PGM is funded by the European Union's Horizon 2020 Research and Innovation Programme, under the Marie Sklodowska-Curie grant agreement n$^\circ$860843.}
\thanks{AH was partially supported by ANR grants ANR-17-CE40-0018 (project CAPPS) and ANR-19-CE40-0014 (project Min-Max).}


\begin{document}

\maketitle

\begin{abstract} 
We prove that for any Borel probability measure $\mu$ on $\R^n$ there exists a set $X\subset \R^n$ of $n+1$ points such that any $n$-variate quadratic polynomial $P$ that is nonnegative on $X$ (i.e. $P(x)\geq 0$, for every $x \in X$) satisfies $\mu\{P\geq 0\}\geq \frac{2}{(n+1)(n+2)}$. 
We also prove that given an absolutely continuous probability measure $\mu$ on $\R^n$ and $D\leq 2k$, for every $\delta>0$ there exists a set $X\subset \R^n$ with $|X|\leq \binom{n+2k}{n}-n-1$ such that any $n$-variate polynomial $P$ of degree $D$ that is nonnegative on $X$ satisfies $\mu\{P\geq 0\}\geq \frac{1}{\binom{n+2k}{n}+1} - \delta$. These statements are analogues of the celebrated \emph{centerpoint theorem}, which corresponds to the case of linear polynomials. 
\vskip2pt
Our results follow from new estimates on the Carath\'eodory numbers of real Veronese varieties, or alternatively, from bounds on the nonnegative symmetric rank of real symmetric tensors. 
\end{abstract}
 
\section{Introduction}

The celebrated \textbf{centerpoint theorem} of Rado and Birch \cite{rado,birch_on_3N_points} states that \emph{for any probability measure $\mu$ on $\R^n$ there exists a point $c\in \R^n$, such that every closed half-space $H \subset \R^n$ with $c \in H$ satisfies $\mu(H)\geq \frac{1}{n+1}$} (this follows from Helly's theorem, see e.g. \cite{pach2011combinatorial}). In data sciences, points of high Tuckey (or half-space) depth play the role of the median for higher-dimensional data, and the centerpoint theorem states that \emph{every distribution has a point of Tuckey (or half-space) depth at least $\frac{1}{n+1}$}. In the language of convexity, the centerpoint theorem states that \emph{the $(\frac{1}{n+1})$-floating body of any probability measure is not empty} (see e.g.~\cite{tukey1975mathematics, nagy2019halfspace, barany2020random, karasev2010topological} and the references therein for different perspectives and a large body of work around the centerpoint theorem and Tuckey depth).

In this paper, we study a generalization of this theorem where half-spaces are replaced by superlevel sets of polynomials, i$.$e$.$ the loci of points in $\R^n$ satisfying a polynomial inequality. We begin with a motivating example in which we replace half-planes in $\R^2$ by disks and their complements. We consider $N$ points in $\R^2$ and the following questions.

\begin{question}
\label{Q1} 
Is there a point $c$ in $\R^2$ such that, given a polynomial of the form $P(x_1,x_2):= \alpha_0 + \alpha_1 \, x_1 + \alpha_2 \, x_2 + \alpha_3 \, (x_1^2 + x_2^2)$, if $P(c)\geq 0$ then the superlevel set $P(x\geq 0) := \{x\in \R^2: P(x)\geq 0\}$ contains a fixed fraction of the points?
\end{question} 

The answer is no, since for every point $c$ in $\R^2$ we can take a disk centered at $c$ of sufficiently small radius, which does not contain any of the points except maybe $c$. Thus, it encloses at most $\frac{1}{N}$ of the points, which is an arbitrarily small fraction. 

\begin{question}
\label{Q2} 
Are there two points $c_1,c_2$ in $\R^2$ such that given $P$ as before, if $P(c_1)\geq 0$ and $P(c_2)\geq 0$ then the superlevel set $P(x\geq 0)$ contains a fixed fraction of the points?
\end{question} 

\begin{proof}[Answer to Question~\ref{Q2}]
Consider the embedding 
\begin{eqnarray*}
\psi:\R^2 & \xrightarrow{} & \R^3 \\ \nonumber
(x_1,x_2) & \mapsto & (X,Y,Z) = (x_1,x_2,x_1^2 + x_2^2)
\ .
\end{eqnarray*} 
The image of $\psi$ is the paraboloid of revolution defined by $V \equiv X^2 + Y^2 - Z = 0$. 
By the centerpoint theorem in $\R^3$ there exists a point $c$ in $\R^3$ such that every half-space $H\subset \R^3$ defined by
$$ 
\alpha_0 + \alpha_1 X + \alpha_2 Y + \alpha_3 Z \geq 0
$$
containing $c$ also contains at least $N/4$ of the points in the image of $\psi$. 
This centerpoint $c$ might not lie on $V$, but it lies in its convex hull. 
Any line $\ell$ through $c$ intersects $V$ at two points $c_1,c_2$. 
Moreover, any half-space $H$ containing the points $c_1$ and $c_2$ must also contain $c$, and therefore it contains at least $N/4$ of the points. The pullback of $H \cap V$ is precisely the superlevel set of
$$ 
P(x_1,x_2):= \alpha_0 + \alpha_1 x_1 + \alpha_2 x_2 + \alpha_3 (x_1^2 + x_2^2)
\ .
$$
Therefore, Question \ref{Q2} is answered affirmatively. 

\end{proof}

\begin{figure}[!t]
\minipage{0.32\textwidth}
  \includegraphics[width=\linewidth]{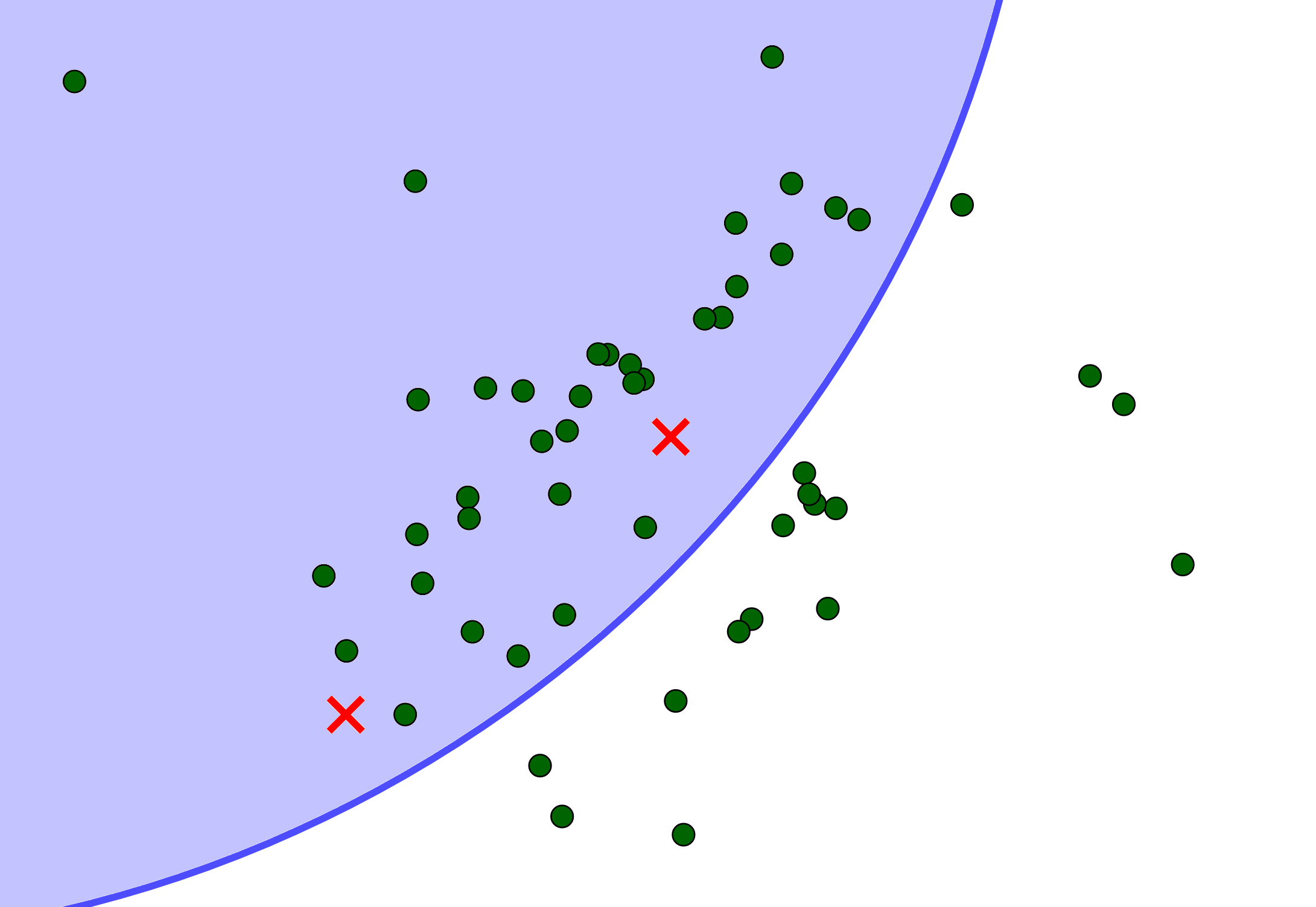}
\endminipage\hfill
\minipage{0.32\textwidth}
  \includegraphics[width=\linewidth]{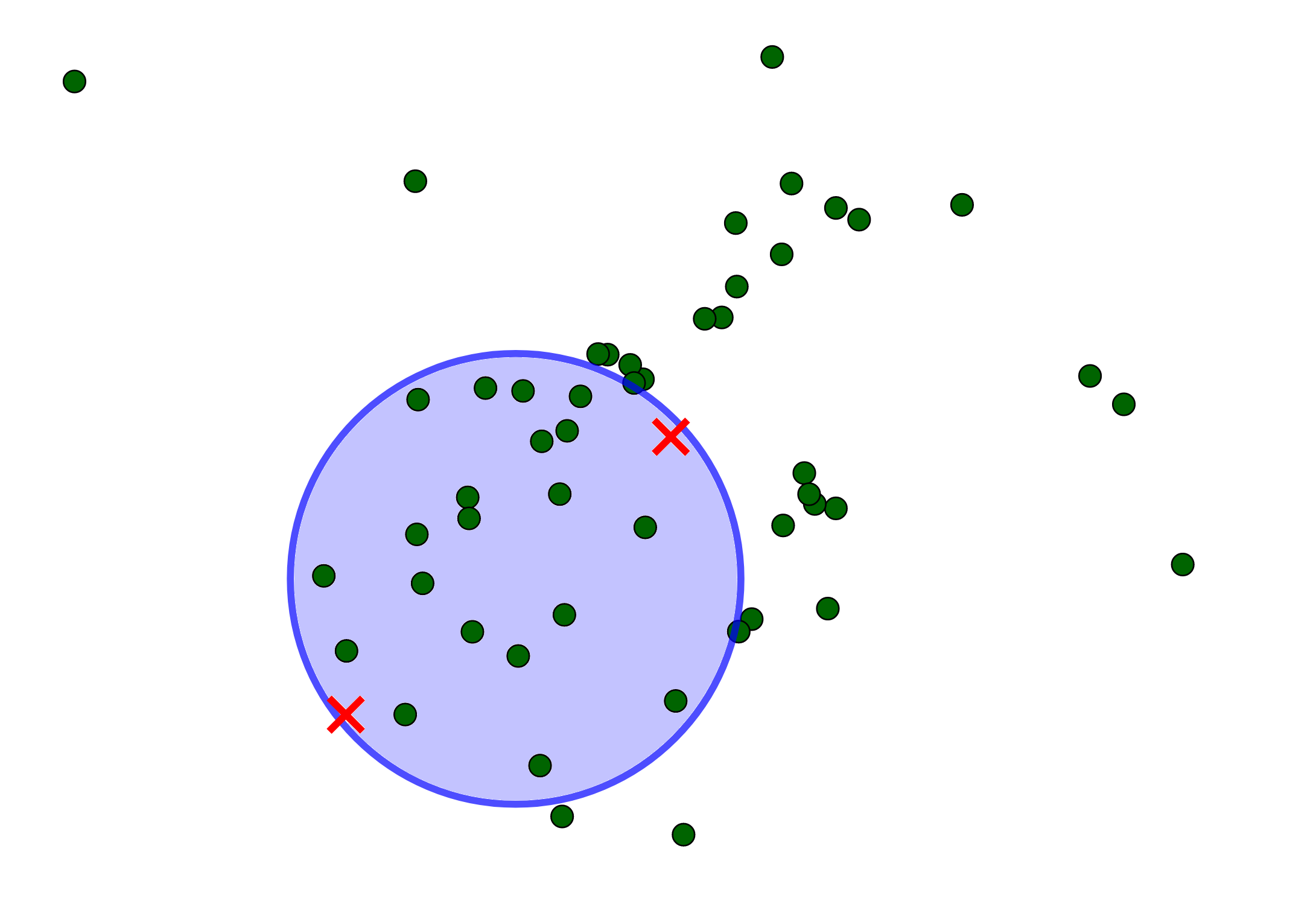}
\endminipage\hfill
\minipage{0.32\textwidth}%
  \includegraphics[width=\linewidth]{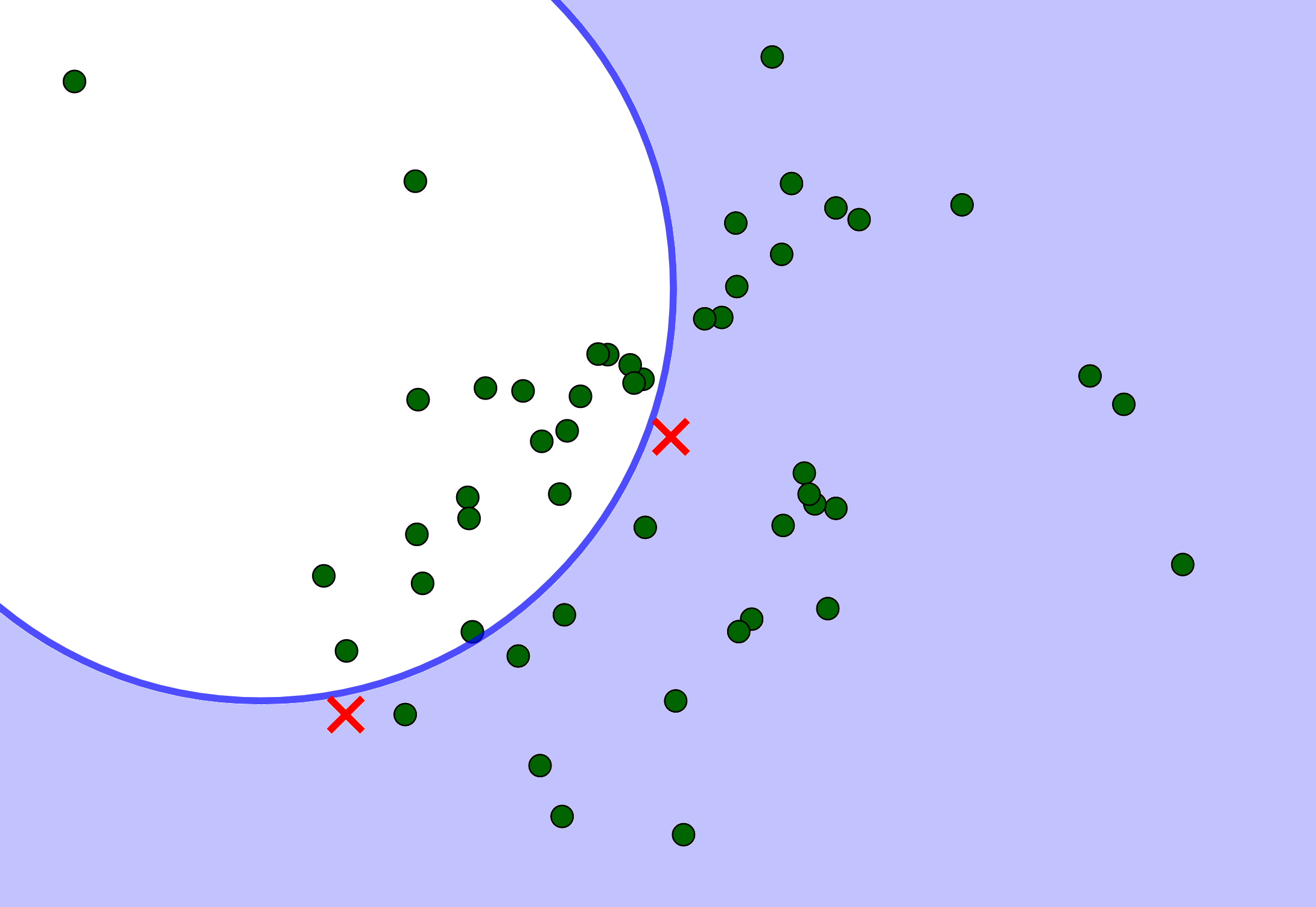}
\endminipage
\caption{Three copies of a set of $N = 50$ points. Every disk and every disk complement that contains the two red crosses contains at least $N/4$ of the points.}
\end{figure}

The example above shows that to derive an analogue of the centerpoint theorem for the superlevel sets of polynomials generated by  $\{1,x,y,x^2+y^2\}$, two points are necessary instead of just one.
The same idea works in any dimension. Namely, for any set of $N$ points in $\R^n$ there are two points such that every ball or complement of a ball that contains this pair also contains a ($\frac{1}{n+2}$)-fraction of the $N$ points. 

More generally, we consider all the polynomials of degree $D$ in $n$ variables, which naturally yields the study of Veronese varieties.  In our statements, we use finite Borel measures in $\mathbb R^n$ which we normalize to probability measures. 
Notice that this implies analogous result about point sets, which corresponds to the case of finite sums of delta masses. 

The following result follows from Theorems~\ref{center} and \ref{caratheodory}. 

\begin{theorem}\label{quadrics}
For any Borel probability measure $\mu$ on $\R^n$ there exists a set $X\subset \R^n$ of $n+1$ points such that, for any quadratic polynomial $P$ satisfying $P(x)\geq 0$ for every $x\in X$, 
\[
\mu\{P \geq 0 \} =
\mu\{x \in \R^n\where P(x)\geq 0\}\geq \frac{2}{(n+1)(n+2)}
\ .
\]
\end{theorem}

Notice that the number of points in Theorem \ref{quadrics} is the least possible. More specifically, we can always find a linear polynomial $L$ vanishing on $n$ points in $\R^n$, so that $-L^2$ is negative outside its zero set. 
If we further assume that $\mu$ is absolutely continuous with respect to the Lebesgue measure, it follows that $\mu\{-L^2\geq 0\}=0$.

\subsection{Weak and strong $(1-\epsilon)$-nets}

Finite sets of centerpoints can be reinterpreted as \emph{weak $(1-\epsilon)$-nets}. 
However, in the classical literature the emphasis has always been on $\epsilon$-nets, where $\epsilon>0$ is small. 
Here we are concerned about $(1-\epsilon)$ close to $1$. 

Let us remind some definitions for measures in $\mathbb R^n$. 

\begin{definition}
Let $F$ be a family of subsets (also called ranges) of $\R^n$, and $\mu$ a probability measure on $\R^n$. We say that $X\subset \R^n$ {\bf is an $\epsilon$-net} for $(\mu,F)$ if for every set $f \in F$, if $\mu(f) > \epsilon$ then $f \cap X\neq \emptyset$. We call a it a {\bf strong} $\epsilon$-net if it lies in the support of $\mu$, and {\bf weak} when there are no additional conditions on its support. 
\end{definition}

The motivating examples for this definition are the family of closed half-spaces in $\R^n$ for strong $\epsilon$-nets, and the family of closed convex sets for weak $\epsilon$-nets (see for instance \cite{pach2011combinatorial, mustafa2017epsilonapproximations}). 

Strong $\epsilon$-nets for half-spaces were introduced in the learning theory literature \cite{VC}, where they remain an important concept closely connected to the Vapnik--Chervonenkis dimension. Their relevance to computational geometry and approximation algorithms was pioneered in \cite{epsilon-nets}. Weak $\epsilon$-nets were introduced in \cite{BaranyFL90} to attack the celebrated halving sets problem. 

The main object of this paper are weak $(1-\epsilon)$-nets for the family of superlevel sets of polynomials as ranges. 
This goal is better understood if we reformulate the classical centerpoint in the language of $(1-\epsilon)$-nets. 

\begin{theorem}[Centerpoint theorem]
\label{theorem: centerpoint as epsilon nets}
Let $\mathcal H$ be the set of half-spaces in $\mathbb R^n$. 
For any Borel probability measure $\mu$ on $\R^n$, then $(\mu,\mathcal H)$ admits a weak $(\frac{n}{n+1})$-net consisting of a single point.
\end{theorem}

\subsection{Veronese varieties and their Carath\'eodory number}
\label{subsection: veronese maps}

For the rest of the paper, we set $m = m(n,D):={n+D \choose n}$ to shorten some formulas. This is the dimension of $\poly(n,D)$, the $\R$-vector space of $n$-variate polynomials of degree at most $D$. The {\bf real Veronese cone}  is the image of the Veronese map denoted by \[\hat{v} \colon \R^{n+1} \to \hat{V}(n,D) \subset \R^{m},\] that sends a point to the vector of all the monomials in $n+1$ variables of degree exactly $D$. 
For instance, when $n=2$ and $D=2$ we have
\[
\hat{v} \colon (w,x,y) \in \R^3 \mapsto (w^2,wx,wy,xy,x^2,y^2)\in \R^6
\ .
\] 

The {\bf real affine Veronese variety} is the image of the map 
\[v \colon \R^n \to V(n,D) \subset \R^{m-1}\] that sends a point to the vector of all the monomials in $n$ variables of degree at most $D$. Equivalently, it corresponds to setting the first coordinate of $\R^{n+1}$ equal to $1$ in $\hat{v}$. In our example, 
\[
v \colon (1,x,y) \equiv (x,y) \in \R^2 \mapsto  (1,x,y,xy,x^2,y^2) 
\equiv 
(x,y,xy,x^2,y^2)\in \R^5
\ .
\]
In particular, $V(n,D)$ can be regarded as contained in the hyperplane $y_0 = 1$, where $y_0,y_1,\ldots ,y_{m-1}$ are the variables in $\R^{m}$. Observe that the map $v$ induces a linear isomorphism  
\begin{eqnarray*} 
\poly(n,D)
& 
\xrightarrow{} 
& 
\poly(m-1,1),
\end{eqnarray*} 
and, in particular, we can identify superlevel sets of polynomials of degree at most $D$ in $\R^n$ with half-spaces in $\R^{m-1}$. 

\vskip5pt

Given $X\subset \R^m$, we denote by $\conv(X)$ the convex hull of $X$ . 
The \textbf{Carath\'eodory number} of $X$ is  the minimal integer $\kappa(X)$ such that every point $x$ in $\conv(X)$ can be written as a convex combination of $\kappa(X)$ points in $X$. 
The Carath\'eodory theorem asserts that $\kappa(X)\leq n+1$. Slightly more generally, the Carath\'eodory number of a point set $X$ is at most the dimension of the space of restrictions of affine functions to $X$. 

\subsection{Symmetric Tensors}

An equivalent perspective on the Veronese cone arises from symmetric tensors. The space $\text{Sym}^D(\R^{n+1})$ of real symmetric tensors of the form $(n+1)^{\times D} = (n+1)\times (n+1)\times \ldots (n+1)$ is isomorphic as a vector space to $\R^{m}$. Namely, the map
\[ 
x \mapsto x^{\otimes D} = x\otimes x\otimes \ldots \otimes x
\]
is related to the aforementioned Veronese map $\hat{v}$ by a linear isomorphism between $\R^m$ and $\text{Sym}^D(\R^{n+1})$, that sends a point in $\hat{V}(n+1,D)$ to a symmetric tensor that can be written as the $D$-th tensor power of some vector in $\R^{n+1}$.

In our example,  
\[
(w,x,y) \mapsto 
(w,x,y) \otimes (w,x,y) \equiv (w,x,y)^T  \cdot (w,x,y)
=
\begin{pmatrix}
w^2& wx & wy \\
wx & x^2 & xy \\
wy & xy & y^2 
\end{pmatrix} 
\ .
\]

The symmetric rank of a symmetric tensor is given by 

\[\text{rank}(v):=\min \{ r \where v = \sum_{i=1}^r \lambda_i \, x_i^{\otimes D} \text{ where } x_i \in \R^{n+1} \text{ and } \lambda_i \in \R \text{ for every } 1\leq i\leq r \}\]

Observe that if we restrict ourselves to nonnegative coefficients $\{\lambda_1,\lambda_2\ldots \lambda_k\}$ that add up to $1$, then we recover the definition of the Carath\'eodory number. 
For this reason, the Carath\'eodory number can also be referred to as the {\bf nonnegative symmetric rank}.  

\subsection{Statements of our results}

We are ready to state our results. The first one gives a lower bound on the number of points of a weak $(1-\epsilon)$-net. We denote by $\poly_+(n,D)$ the range space of superlevel sets of $n$-variate polynomials of degree at most $D$, i$.$e$.$ 

$$
\poly_+(n,D) := 
\{ 
\{P \geq x\} :  P\in \poly(n,D)
\}
\ .
$$

\begin{proposition}
\label{square}
Let $\mu$ be a probability measure on $\R^n$ which is absolutely continuous with respect to the Lebesgue measure.
For every $\epsilon>0$, any $(1-\epsilon)$-net of $(\mu,\poly_+(n,2D))$ has size at least $\binom{n+D}{n}$. 
\end{proposition}

On the other direction we have the following.

\begin{theorem} 
\label{center}
Let $\mu$ be a Borel probability measure in $\mathbb R^n$. There exists a weak $\left(1-\frac{1}{m(n,D)}\right)$-net of $(\mu,\poly_+(n,D))$ of size at most $\kappa(V(n,D))$. If $D$ is even and $\mu$ is absolutely continuous with respect to the Lebesgue measure, then for every $\delta>0$ there exists a weak $\left(1-\frac{1}{m(n,D)+1}+\delta\right)$-net of $(\mu,\poly_+(n,D))$ of size at most $\kappa(\hat{V}(n+1,D))$.
\end{theorem}

We remark that for a measure supported on a finite set of points, one can show that the weak net $X$ can be chosen to satisfy a slightly stronger condition: \emph{for every degree polynomial $P$ of degree $D$, if $P(x) \geq 0$ for every  $x \in X$  (note the inequality $\geq$ here), then $\mu\{P\geq 0\}\geq \frac{1}{m(n,D)}$.}  It is easy to modify our proof to obtain a strong $(1-\epsilon)$-net (with the same $(1-\epsilon)$ as in the theorem) of cardinality at most $\binom{n+D}{n}$. 
The only difference is that one has to apply the Carath\'eodory theorem to the support of the push-forward of the measure in place of the whole Veronese variety, loosing a chance to improve the trivial $\binom{n+D}{n}$ bound on the Carath\'eodory number. 

\vskip.5cm
 
Our next task is to estimate the Carath\'eodory numbers of the Veronese varieties. From Theorem~\ref{center} and Proposition~\ref{square} we obtain a lower bound $\kappa(V(n,2k))\geq \binom{k+n}{n}$, and from the Carath\'eodory theorem $\kappa(V(n,D))\leq \binom{n+D}{n}$.  We remark that giving better estimates on these numbers is a well know question raised by several authors (e.g. \cite{barvinok2002}, \cite{barany2012notes} \cite{Dio_2018} and page 10 of \cite{gromov}). Technically the bounds provided in the next theorem are our most interesting result.  

\begin{theorem}
\label{caratheodory} 
Let $n,k$ be positive integers. The Carath\'eodory numbers of the Veronese varieties satisfy the following:
\begin{enumerate}
\item $\kappa(V(n,2))= n+1$
\item $($di Dio and Kummer, \cite{di_Dio_2021}$)$ $\kappa(V(n,2k)) \geq \binom{2k+n}{n} - n \binom{k+n}{n} + \binom{n}{2}$ 
\item $\kappa(\hat V(n+1,2k)\leq \binom{2k+n}{n}-n-1$
\item $\kappa( V(n,2k) ) \leq \kappa( V(n,2k+1) ) \leq \kappa( V(n,2k+2) )$
\end{enumerate}
\end{theorem}

Item (1) of Theorem~\ref{caratheodory} follows from the spectral theorem for positive semidefinite matrices and is probably known. Item (2) follows from Noga Alon's Combinatorial Nullstellensatz \cite{alon1999}. This result turned out to be recently discovered in \cite{di_Dio_2021}, but we provide a direct geometric argument that  does not require familiarity with commutative algebra and Hilbert polynomials as in \cite{di_Dio_2021}. Item (3) relies on a necessary conditions for a hyperplane to be a supporting hyperplane of $\conv(\hat{V}(n+1,D)$ which might be of own interest.  Item (4) is a trivial observation: $V(n,D)$ is a linear projection of $V(n,D+1)$.

\vskip.5cm

From Theorems~\ref{caratheodory} and \ref{center} we derive Theorem~\ref{quadrics}, as well as an analogue for every degree. 

\begin{corollary} 
Let $\mu$ be probability measure on $\R^n$ which is absolutely continuous with respect to the Lebesgue measure, and let $D \leq 2k$. For every $\delta>0$, there exists a weak $(1-\frac{1}{m(n,2k)+1}-\delta)$ net for $(\mu, \poly_+(n,D))$ with at most $\binom{2k+n}{n}-n-1$ points.
\end{corollary}

\subsection{Discussion}

\subsubsection{Neighborliness of the Veronese}

The case $\kappa(V(1,D))$ is very interesting on its own right. The variety $V(1,D)$ is also called the moment curve, and is very important in polytope theory. The crucial property that the moment curve of degree $D=2k$ satisfies is that it is $k$-\emph{neighborly}, that is, for every $k$ points on the moment curve there is a support hyperplane that touches the convex hull of the moment curve exactly at these points (see e.g. \cite{barvinok2002}).

This observation can be extended to $n>1$ as follows. If $X=\{x_i\}_{i=1}^k$ is an arbitrary set of $k$ points, then $P(x)=\Pi_{i=1}^k (x-x_i)^2$ is a degree $D=2k$ polynomial that vanishes exactly at $X$ and is positive elsewhere, so it lifts to a support hyperplane of $V(n,D)$ touching $V(n,D)$ at the Veronese images of these points. In fact, for $D=2k$ the Veronese variety $V(n,D)$ is almost $\binom{n+k}{k}$-neighborly in the following sense: for any $\binom{n+k}{k}$ points $x_1,x_2,\ldots x_{m(n,k)}\in \R^n$, the polynomial $-Q^2$ such that $Q(x_i)=0$ for all $i$, is nonpositive and lifts to a support hyperplane of $V(n,D)$ touching $V(n,D)$ at these points and probably some other points (this is a version of being neighborly).

\subsubsection{Tensor ranks}

The notion of rank used here belongs to a large body of work where there has recently been lots of activity from very different perspectives. Matrix rank can be defined in multiple ways, but when we move to general tensors there are several different possible notions of rank of interest. From an algebraic geometry perspective we might consider $X \subset \P_k^n$ to be a projective variety over the field $k$ that is non-degenerate, i$.$e$.$ it does not lie in a hyperplane. In particular, the projective span of $X$ is $\P_k^n$. Denote by $\hat{X} \subset k^{n+1}$ the affine cone over $X$. Given a non-zero vector $v$ in $k^{n+1}$ (or equivalently a projective point $v$ in $\P_k^n$), we define the rank of $v$ with respect to $\hat{X}$ (or $X$) as  
$
\rank_X(v):=\min \{ r \where v = \sum_{i=1}^r x_i \text{ where } x_i \in \hat X \text{ for every } 1\leq i\leq r \} 
\ .
$  
This notion of rank is very general. It captures matrix and tensor rank when $X$ is the Segre embedding of respectively $\P_k^{n_1}\times \P_k^{n_2}$ and $\P_k^{n_1}\times \ldots \times \P_k^{n_m}$; symmetric tensor rank when $X$ is the Veronese embedding of $\P_k^n$; and antisymmetric tensor rank when $X$ is the Pl\"ucker embedding of the Grassmannian of $m$-dimensional subspaces of $\P_k^n$. Notice that, as $Y$ is non-degenerate, this rank is always finite. 
There are several other interesting notions of rank (slice rank, analytic rank, Schmidt rank, etc.) that have been given attention lately in the literature with applications that go from statistics and neural networks to combinatorial analysis and additive combinatorics. 


Item (2) of Theorem~\ref{caratheodory} is somehow disappointing. As explained in §\ref{subsection: veronese maps} before, the Carath\'eodory number of $\hat{V}(n+1,D)$ is the maximum nonnegative symmetric rank of a symmetric tensor in the convex hull of the locus of rank one tensors. 
Without the nonnegativity condition (allowing to sum tensor powers with signs) a theorem in \cite{general_rank_1} shows that the real symmetric rank is not greater than twice the complex symmetric rank. The complex symmetric rank, in turn, is understood by the celebrated Alexander--Hirschowitz theorem \cite{alexander1995polynomial}: except for a finite number of cases the maximum complex symmetric rank is $\lceil \frac{1}{n+1}\binom{n+D}{n}\rceil$; in the exceptions, it is $\lceil \frac{1}{n+1}\binom{n+D}{n}\rceil + 1$.

We will explain these algebraic results in more detail, and derive the inequality $\kappa(\hat{V}(n+1,D)) \leq \frac{2}{n+1}\binom{n+D}{n}$ for odd $D$ in the last section. What is important to remark now is that this behavior is irrelevant to obtain upper bounds on $\kappa(V(n,D))$ (in the affine case) for both $D$ even and odd. Moreover, the lower bound of item (2) in Theorem~\ref{caratheodory} shows that, contrary to our original hope, as $D\to \infty$ the multiplicative improvement over the trivial bound given by the Carath\'eodory theorem is asymptotically like $\left(1-\frac{n}{2^n}\right)$. If we fix $D$ and let $n$ go to infinity the situation is even more dire, there is no asymptotic improvement over Carath\'eodory's bound.  

As we mentioned before, item (2) of our Theorem~\ref{caratheodory} is not new. The work \cite{di_Dio_2021} proves it with a less elementary argument using the Hilbert polynomials. Also, notice that \cite{di_Dio_2021} has a similar bound for the odd degree case, but considering the Veronese image of the cube $[0,1]^n$ instead of the whole $\mathbb R^n$. It seems to us that items (1) and (3) of our Theorem~\ref{caratheodory} have implications for the moment problem studied in \cite{di_Dio_2021}.

\section{Weak $(1-\epsilon)$-nets}
\label{section:nets}

We start proving that the \emph{centerpoint theorem} is equivalent to Theorem \ref{theorem: centerpoint as epsilon nets}. 

\begin{proof}[Proof of Theorem \ref{theorem: centerpoint as epsilon nets}]
For contradiction, assume that $c$ is a centerpoint which is is not a weak $(\frac{n}{n+1})$-net. Then, there exists a closed half-space $H$ such that $\mu(H) > \frac{n}{n+1}$ but $c \notin H$. This means that $c \in \R^n \setminus H$ and $\mu(\R^n \setminus H) < \frac{1}{n+1}$, which is a contradiction. Similarly, if $c$ is a weak $(\frac{n}{n+1})$-net but not a centerpoint, there exists an open half-space $H$ such that $c \in H$ but $\mu(H)<\frac{1}{n+1}$. Hence $\mu(\R^n \setminus H) > \frac{n}{n+1}$, yet $c \notin \R^n \setminus H$. Therefore, $c$ is not a weak $(\frac{n}{n+1})$-net.
\end{proof}

Also observe that if $\mu$ is an absolutely continuous measure on $\mathbb R^n$, then its Veronese pushforward $v_\#(\mu)$ vanishes on hyperplanes (as $\mu$ vanishes on non-trivial algebraic sets). Hence, for absolutely continuous measures one may consider intersections of either open or closed half-spaces with $V(n,D)$ not affecting its measure.

\begin{proof}[Proof of Proposition~\ref{square}]
We begin with the lower bound. Let $\mu$ be an absolutely continuous measure on $\mathbb R^n$, and let $X = \{ c_1, \ldots , c_{m-1} \}$ be any set of $ m-1 = {n+D \choose n}-1$ points. Consider a polynomial $Q$ of degree $D$ that vanishes on all the points in $X$, which exists since the evaluation linear map
\begin{eqnarray*}
\poly(n,D) & \xrightarrow{} & \R^{m-1} \\
P & \mapsto & (P(c_1) \, , \, \ldots \, , \, P(c_{m-1}))
\end{eqnarray*}
has nontrivial kernel. Consider the polynomial $-Q^2+\eta$ for an arbitrarily small $\eta>0$. In particular, $-Q^2(x)+\eta>0$ for each $x \in X$. Since $\mu$ is absolutely continuous, given $\epsilon>0$ we can choose $\eta$ such that $\mu\{-Q^2(x)+\eta \geq 0\}<\epsilon$. Hence, $X$ is not a $(1-\epsilon)$-net of $(\mu,\poly_+(n,2D))$.
\end{proof}

Notice that for $D=2$ the previous bound yields a lower bound of $n+1$ for the size of $(1-\epsilon)$-nets which matches our upper bound (Theorems \ref{center} and \ref{caratheodory}). 

Let us now pass to the upper bound. In the following proof, we set $\kappa:=\kappa(V(n,D))$.

\begin{proof}[Proof of the first statement of Theorem \ref{center}]
We begin pushing forward the measure $\mu$ through the affine Veronese map $v:\R^n \xrightarrow{} \R^{m-1}$, and applying the centerpoint theorem to the push-forward measure $v_\#(\mu)$ in $\R^{m-1}$ to find a centerpoint $c \in \R^{m-1}$. In particular, every half-space that contains $c$ has measure at least $\frac{1}{m}$. Clearly $c \in \conv(V(n,D))$, as otherwise by the hyperplane separation theorem there exists a half-space containing $c$ disjoint to the support of the measure. 
Now using the definition of the Carath\'eodory number we obtain a point set 
\[
Y:=\{y_1,y_2\ldots y_{\kappa}\}\subset V(n,D)
\] 
so that we can write $c = \lambda_1 y_1 + \ldots + \lambda_\kappa y_\kappa$ for some $\lambda_i\geq 0$ such that $\lambda_1 + \ldots + \lambda_\kappa = 1$.  We set $x_i=v^{-1}(y_i)$, and $X=\{x_1,x_2\ldots x_\kappa\}$. 

We claim that $X$ is a required net. Namely, let $P$ be a polynomial of degree $D$ such that $P(X) > 0$ for every $x_i \in X$. The semialgebraic set defined by $P\geq 0$ is transformed by the affine Veronese embedding $v$ to $V(n,D)\cap H$, where $H\subset \R^m$ is the hyperplane defined by the coefficients of $P$ in the standard basis as explained in §\ref{subsection: veronese maps}. By definition, $H$ contains all the $y_i$'s, and therefore $c$. Hence, we find $v_\#(\mu)(H) = \mu\{P\geq 0\} \geq \frac{1}{m}$, which concludes the proof. 
\end{proof}

\begin{proof}[Proof of the second statement of Theorem \ref{center}]
If we identify $\R^n$ with the hyperplane $x_0 = 1$ in $\R^{n+1}$, we find $V(n,D) \subset \hat{V}(n + 1, D) \subset \R^{m(n,d)}$.
In particular, since the Veronese embedding identifies the variable $y_0$ in $\R^{m(n,d)}$ with the monomial $x_0^D$, the Veronese variety $V(n,D)$ is contained in the affine hyperplane $A$ defined by $y_0=1$.

Let us push forward $\mu$ to $V(n,D)$ and find a centerpoint $c$ of this measure in the ambient $\mathbb R^m$, and let $k = {\kappa(\hat{V}(n+1,D))}$. 
We can find a finite set 
\[
Y:=\{y_1,y_2\ldots y_k\}\subset \hat{V}(n+1,D)
\] 
such that  
$
c = \lambda_1 \, y_1 
+ 
\ldots 
+
\lambda_k \, y_k 
$ 
for some $\lambda_i \geq 0$ satisfying $\lambda_1 + \ldots + \lambda_k = 1$. 
As before, any half-space $H \subset \mathbb R^m$ that contains $Y$ has measure at least $\frac{1}{m+1}$. 

If the first entry of all the points in $Y$ is nonzero, then we can renormalize the convex combination to ensure the $y_i$'s lie in $A$ and therefore in $V(n,D)$, concluding the proof. 
However, this may not be the case since some of the $y_i$ may belong to the hyperplane $A$ defined by $y_0=0$. 
To deal with this situation, let  
Let 
$$
Y^j = 
\{ 
y_1^j , \ldots , y_k^j
\}
$$
define a sequence of sets in $\hat{V}(n+1,D)\cap \{ y_0 > 0\}$ satisfying that 
$
y_i^j \xrightarrow{j\to \infty} y_i
$
for each $0 \leq i \leq k$. 
In particular, the sequence of points defined by
\begin{equation}
\label{eq: convex combination in sequence}
c^j = 
\lambda_1 \, y_1^j 
+ 
\ldots 
\lambda_k \, y_k^j 
\end{equation}
converges to $c$ as $j\to \infty$. 
Moreover, renormalizing \eqref{eq: convex combination in sequence} we can assume that $y^j_i\in V(n,D)$ for every $1 \leq i \leq k$. 

Let us show that, for every $\delta>0$, there exists some $j(\delta,\mu)$ such that $Y^j$ is a weak $\left(\frac{m}{m+1}+\delta\right)$-net. To see this, fix a $\delta>0$ and assume that for every $j$, there exists a degree at most $D$ polynomial $P_j$ such that $P_j(Y^j)>0$ and 
\begin{equation}
\label{measure-too-small}
\mu\{P_j\ge 0 \}<\frac{1}{m+1}-\delta.
\end{equation}
The polynomial $P_j$ is nonzero, and its superlevel set $\{P_j \geq 0\}$ is invariant if we multiply it by a positive number.  
Hence, we may assume that $P_j$ belongs to a compact unit sphere in the space of polynomials. 
We can thus extract a subsequence converging to some polynomial $P_-$ uniformly on compacta. Going to the limit in $P_j(Y^j)\ge 0$ we obtain $P_-(Y)\geq 0$. 
We know \eqref{measure-too-small}, and need to infer that 
\[
\mu \{P_-\ge 0\}\leq \frac{1}{m+1}-\delta
\]
to have a contradiction with the choice of the centerpoint $c$ and $Y$. 

Assume the contrary, i$.$e$.$ $\mu\{P_-\ge 0\} > \frac{1}{m+1}-\delta$. From the absolute continuity of $\mu$ it follows that $\mu$ is zero on algebraic sets and $\mu\{P_- > 0\} > \frac{1}{m+1}-\delta$. Define the sets $X_j\subseteq \mathbb R^n$ by the inequality $P_j > 0$ and the set $X_-$ by $P_- > 0$. From the pointwise limit $P_j\to P_-$ we readily obtain the following property of indicator functions
\[
\chi_{X_-} \leq \liminf_{j\to \infty} \chi_{X_j}. 
\]
Since for every $j$ we have $\mu(X_j) < \frac{1}{m+1}-\delta$, applying the Fatou lemma we obtain
\[
\int_{\mathbb R^n} \chi_{X_-}d\mu = \mu(X_-)\leq \frac{1}{m+1}-\delta, 
\]
which is a contradiction.
\end{proof}

\section{Carath\'eodory numbers of Veronese varieties}

Now we pass to the bounds on the Carath\'eodory numbers.

\subsection{Case $D=2$ of Theorem~\ref{caratheodory}}

We begin bounding the Carath\'eodory number in the degree two homogeneous case of Theorem \ref{caratheodory}.
 
\begin{proof}[Proof of the Carath\'eodory number of the Veronese cone of degree $2$]
Let $x$ be a nonzero vector and $x^T$ its transpose.
 Put $A(x):=x \otimes x = x x^T$, i.e.  $A(x)_{ij}=x_ix_j$ which is a positive symmetric matrix. 
Notice that the Veronese cone $\hat{V}(n+1,2)$ is the set of rank one $(n+1)\times (n+1)$ symmetric matrices, except for its apex $(0,\ldots , 0)$ which is identified with the zero matrix. 
Any positive combination of such matrices is positive semidefinite, and by the spectral theorem if $M$ is positive semidefinite matrix, we can decompose it as 
\[
M=\sum_{i=1}^{n+1}  E_i
\]
where $E_i= c_i \, x_i \otimes x_i$ and $\{x_1,x_2\ldots x_{n+1}\}$ is an orthonormal eigenbasis of $M$ with corresponding eigenvalues $\{c_1,c_2 \ldots c_{n+1}\}$. We can interpret this as a bound of $n+1$ on the Carath\'eodory number, rewriting 
\[
M=\sum_{i=1}^{k} \frac{c_i}{c} (\sqrt{c} x_i)\otimes (\sqrt{c} x_i),
\]
where $c:=\sum_i c_i$. 
Since there exist $(n+1)\times (n+1)$ symmetric matrices of rank $n+1$, this upper bound on the Carath\'eodory number is in fact attained.
\end{proof}

Now we show the equality of the Carath\'eodory number for affine Veronese varieties.

\begin{proof}[Proof of the Carath\'eodory number of the affine Veronese variety of degree $2$, Theorem~\ref{caratheodory}.1]

Embed $\mathbb R^n$ to $\mathbb R^{n+1}$ with the additional coordinate $x_{0}=1$. 
Then the Veronese variety $V = V(n,2)$ consists of rank one $(n+1)\times (n+1)$ symmetric matrices $x\otimes x$ for some nonzero $x = (1,x_1,\ldots,x_n$). 

Similarly to the above argument, for any $M$ in the convex hull of $V$ there exist a sum of $k\le n+1$ rank one positive matrices
\[
M = E_1 + \dots + E_k.
\]
If $M$ is full rank and $k=n+1$ then there are a lot of such decompositions transformed by $M$-orthogonal (keeping the quadratic form $M$ invariant) rotations of $\mathbb
R^{n+1}$ one into another. Because of this, one may also assume that the corner element $(E_i)_{0,0}$ is positive for every $i$. 
In order to achieve this, one just need to rotate so that the hyperplanes 
\[
\ker E_i = \{ x^T E_i x = 0 \}, \quad i=1,\ldots,n+1,
\]
do not contain the ``vertical'' vector $\bar{z}$ defined by $\bar{z}_{0}=1$ and $\bar{z}_i=0$ for $i=1,\ldots, n$.

If $A$ is not full rank then the decompositions into $k=\rank M$ matrices of rank $1$ are also transformed one into another by $M$-orthogonal rotations that fix $\ker M$ and rotate the $k$-dimensional quotient $\mathbb R^{n+1}/\ker M$. Observe that for the ``vertical'' vector we have $\bar{z}\not\in\ker M$ (equivalently $M_{0,0}\neq 0$). Hence it is possible to apply an $M$-orthogonal rotation and choose the $E_i$ so that $\bar{z}$ does not belong any of the hyperplanes $\ker E_i$, $i=1,\ldots,k$, since the intersection of those hyperplanes is precisely $\ker A$ and their images in $\mathbb R^{n+1}/\ker M$ can be rotated arbitrarily.

After that it is possible to take out the positive corner elements of
the $E_i$ and write
\[
M = c_1 x_1\otimes x_1 + \dots + c_k x_k\otimes x_k
\]
with positive $c_i$ and vectors $x_i$ such that $(x_i)_{0} = 1$. Since
$M\in H = \{M_{0,0}=1\}$ and all $x_i\otimes x_i$ are in the hyperplane
$H$ then the sum of $c_i$ must be $1$. Hence $M$ is a convex combination
of at most $n+1$ elements of $V$.
\end{proof}

\subsection{Case of even degree $D\ge 4$ of Theorem~\ref{caratheodory}}

\begin{proof}[Proof of the lower bound for even $D=2k$, Theorem~\ref{caratheodory}.2]
Let
\[
F(x) = (x-1)\dots (x-k)
\]
be a univariate polynomial of degree $k$ with $k$ distinct real roots. Consider the multivariate polynomial of degree $2k$
\[
P(x_1,\ldots, x_n) = F(x_1)^2 + \dots + F(x_n)^2.
\]
Obviously, $P$ is nonnegative and its zero set $M$ is the finite product $\{1,\ldots, k\}^n$ of size $k^n$. 

The Veronese image $v(M)$ is an intersection of the full Veronese image $V(n, D)$ with its support hyperplane, corresponding to $P$. Hence for the Carath\'eodory number we have $\kappa(V(n,2k)) \ge \kappa (v(M))$. Since $v(M)$ is a finite set, its Carath\'eodory number equals the dimension of its affine hull plus one, $\dim v(M) + 1$. This in turn equals the dimension of the space of restrictions of polynomials of degree at most $D$ to $M$.

In order to understand the dimension of the space of restrictions, let us analyze the kernel of the restriction map, that is the space $\mathcal Z(M)$ of polynomials of degree at most $2k$ vanishing on $M$. Any $n$-variate polynomial $Q\in\mathcal Z(M)$ can be written as
\[
Q(x) = S_1(x) F(x_1) + \dots + S_n(x) F(x_n),
\]
using Noga Alon's Combinatorial Nullstellensatz \cite[Theorem~1.1]{alon1999}, where $\deg S_i\le \deg Q - k$ for any $i$. When $\deg Q\le 2k$, we have $\deg S_i\le k$ for any $i$.

The dimension of all possible combinations of the $S_i$ is at most $n\cdot m(n, k) := n \binom{k + n}{n}$. Moreover, the map 
\[
(S_1,\ldots,S_n)\mapsto S_1(x) F(x_1) + \dots + S_n(x) F(x_n)
\]
obviously has a kernel spanned by the following $\binom{n}{2}$ linearly independent vectors of polynomials whose nonzero elements occupy positions $i<j$:
\[
(0,\ldots, F(x_j),\ldots, - F(x_i), \ldots, 0).
\] 

Hence the kernel of the restriction map for polynomials of degree $2d$ has dimension at most $n\binom{n+k}{k}-\binom{n}{2}$. Then for the degree of the image of the restriction map we obtain
\[
\kappa(V(n,2k)) \ge m(n, 2k) - n\cdot m(n, k) + \binom{n}{2} = \binom{2k+n}{n} - n \binom{k+n}{n} + \binom{n}{2}.
\] 
When $D=2k\to \infty$, this is asymptotically 
\[
\kappa(V(n,2k)) \ge \frac{(2k)^n}{n!} - n \frac{k^n}{2^n n!} = \left(1 - \frac{n}{2^n}\right) \frac{(2k)^n}{n!} \sim \left(1 - \frac{n}{2^n}\right) m(n, 2k).
\]
\end{proof}

\begin{proof}[Proof of the upper bound for even $D=2k$, Theorem~\ref{caratheodory}.3]

We may follow the way of finding the Carath\'eodory number of the moment curve ($V(1,D)$ in our notation) from \cite{barvinok2002}. Consider $\hat V(n+1, D)$ for even $D$. This is a cone over its section by the hyperplane
\[
(x_0^2 + x_1^2 + \dots + x_n^2)^{D/2}=1
\]
that corresponds to restricting the homogeneous polynomials to the unit sphere $\mathbb S^n$. We may focus on estimating the Carath\'eodory number of this section $\hat v(\mathbb S^n)$, since it is the same as the Carath\'eodory number of the whole cone. This point of view has an advantage that the homogeneous Veronese map $\hat v$ is therefore considered on the compact sphere, its Veronese image being also compact and its convex hull being a convex compact set.

Now take a point $\xi \in \conv(\hat v(\mathbb S^n))$. Take any point $x\in \mathbb S^n$ and consider the largest $t$ such that
\[
\xi - t \hat v(x)\in \conv \hat v(\mathbb S^n).
\] 
This maximum is attained from compactness of the convex hull. Set $\eta = \xi - t \hat v(x)\in \partial \conv \hat v(\mathbb S^n)$ for the maximal $t$. It remains to express $\eta$ as a combination of not too many points of $\hat v(\mathbb S^n)$. The point $x$ then will a convex expression in $v(x)$ and the other points. 

Since $\eta \in\partial \conv \hat v (\mathbb S^n)$, by the Hahn--Banach theorem there exists a support affine functional to the convex hull $\conv(\hat v(\mathbb S^n))$ at $\eta$. In view of the homogeneity, we may consider it a linear functional on the image $\hat v(\mathbb R^{n+1})$, that is, a homogeneous polynomial $P$ in of degree $D$ in $n+1$ variables. The support property then decodes as: $\min P = 0$ and the zero set
\[
M = \{ x\in\mathbb S^n\ |\ P(x) = 0\}
\] 
has the property that $\eta\in \conv \hat v(M)$.

Therefore, if we want to express $\eta$ as a convex combination of not too many elements of $\hat v(\mathbb S^n)$ then we have to try to express it as a convex combination of not too many elements of $\hat v(M)$. Note that any estimate $\dim \conv \hat v(M)\le N$ with the use of the standard Carath\'eodory theorem implies that this expression will involve at most $N+1$ points. Note also that 
\[
\dim \conv \hat v(M) = \binom{n+D}{n} - \dim \mathcal Z(M) - 1,
\]
where $\mathcal Z(M)$ is the space of homogeneous polynomials in $n+1$ variables of degree $D$ that vanish on $M$. This is so because in the image of $\hat v$ the degree $D$ polynomials correspond to the linear functions (restricting to affine functions on the affine span of $v(M)$). Hence we have 
\[
\kappa\left(\hat v(M)\right) \le \binom{n+D}{n} - \dim \mathcal Z(M)
\]
and 
\[
\kappa\left(\hat V(n+1, D)\right) \le \binom{n+D}{n} - \dim \mathcal Z(M) + 1.
\]

Hence we need to give lower bounds on $\dim \mathcal Z(M).$ Evidently, $P\in\mathcal Z(M)$ and therefore $\dim \mathcal Z(M) \ge 1$. But this is trivial and only provides the trivial upper bound $\binom{n+D}{n}$ on the Carath\'eodory number of $\hat V(n+1, D)$.

We may note that, since $M$ is the set where the minimum of $P$ is attained, all the $(n+1)^2$ combinations
\[
x_i\frac{\partial P}{\partial x_j}
\]
also vanish on $M$ (and are homogeneous polynomials of degree $D$). Since there is the Euler relation
\[
P = \frac{1}{D} \sum_{i=0}^n x_i\frac{\partial P}{\partial x_i},
\]
the lower bound on $\mathcal Z(M)$ obtained this way is at most $(n+1)^2$. In fact, it can be worse, since for a polynomial as bad as
\[
P(\bar x) = x_0x_1\dots x_n
\]
we have for every $i=0,1,\ldots,n$
\[
x_i\frac{\partial P}{\partial x_i} = P,
\]
which reduces the number of polynomials vanishing on $M$.

On the positive side, we may check a linear dependence of $x_i\frac{\partial P}{\partial x_j}$ with a fixed $j$ and varying $i$. The dependence
\[
a_0 x_0 \frac{\partial P}{\partial x_j} + \dots + a_n x_n \frac{\partial P}{\partial x_j} = (a_0 x_0 + a_1 x_1 + \dots + a_n x_n) \frac{\partial P}{\partial x_j} = 0
\]
implies that either
\[
a_0 x_0 + a_1 x_1 + \dots + a_n x_n = 0
\]
(which is impossible for nontrivial coefficients $a_i$), or $\frac{\partial P}{\partial x_j}$ is identically zero. The latter may happen for a particular $j$, but cannot happen for all $j$, since $P$ is nonzero and non-constant (since it is homogeneous). Hence with some $j$ we find $n+1$ linearly independent
\[
x_0\frac{\partial P}{\partial x_j}, \ldots, x_n\frac{\partial P}{\partial x_j}\in \mathcal Z(M).
\]
Consider $P\in \mathcal Z(M)$, if it is linearly independent of those $n+1$ polynomials then we obtain
\[
\dim\mathcal Z(M)\ge n+2\Rightarrow \kappa \left( \hat V(n+1,D) \right) \le \binom{n+D}{n} - n - 1.
\]
If this is not the case, assume without loss of generality that $j=0$. Let
\[
(a_0 x_0 + a_1 x_1 + \dots + a_n x_n) \frac{\partial P}{\partial x_0} = P.
\]
If $a_j = 0$ then restricting to the lines of constant $x_1,\ldots, x_2$ we obtain a differential equation whose solution is an exponent, which cannot happen for a polynomial $P$.

Otherwise we change the coordinates so that $x_0 + a_1/a_0 x_1 + \dots + a_n/a_0 x_n$ is the new $x_0$, other coordinates remaining the same. Then the equation becomes
\[
a x_0 \frac{\partial P}{\partial x_0} = P,
\]
where we set $a=a_0$ for brevity. Any solution of this is 
\[
P(x) = x_0^{1/a} Q(x_1,\ldots,x_2).
\]
Since $P$ is a nonnegative polynomial, $1/a$ has to be an even positive integer and $Q$ must be a nonnegative polynomial. Then we observe that $R = x_1 \frac{\partial Q}{\partial x_j}$ vanishes on $M$, has degree at most $D-2$, and is nonzero for at least one value of $j=1,\ldots, n.$ All $\binom{n+2}{2}$ combinations $x_ix_j F$ ($i\le j$) are linearly independent, have degree $D$, and vanish on  $M$, thus improving the bound to
\[
\dim\mathcal Z(M)\ge \binom{n+2}{2} \ge n+2.
\]
Hence we have the result in any case.
\end{proof}

\subsection{Odd degree}

In the affine case our bounds for the $D$ odd (Theorem~\ref{caratheodory}.4) follow from the fact that linear projections $V(n,D)\to V(n,D-1)$ preserve convex combinations. In the homogeneous odd case (not stated in Theorem~\ref{caratheodory} since it is irrelevant to $(1-\epsilon)$-nets) the bound follows from the results of \cite {general_rank_1} and \cite{alexander1995polynomial}. Here we provide the statements of those results for completeness.

The symmetric tensor rank of a symmetric tensor $P$ of degree $D$ in $n$ variables is the smallest number $r$ such that we can write $P$ as a linear combination of $r$ rank $1$ tensors. In terms of polynomials, this correspond to $D$-powers of linear forms. The Waring problem consists in determining the maximal symmetric rank of a degree $D$ tensor. 

\begin{theorem}\cite{alexander1995polynomial}
\label{ah}
The symmetric rank of a generic $D$-homogeneous complex symmetric tensor on $(n+1)$ variables is $\lceil \frac{1}{n+1} m(n,D)\rceil$ except if $D=2$, or $(n,D)$ is one of the following cases $(3,4), (4,4), (5,4), (5,3)$.
\end{theorem}

Almost every tensor is contained on an open set in which the rank is constant, we call a rank $r$ a {\bf typical rank} if there exist an open set $A$ such that the rank of every tensor in $A$ is $k$. Over the complex numbers there exists a unique typical rank because the discriminant variety has complex dimension $1$. Over the reals there are several typical ranks. 

\begin{theorem}\cite{general_rank_1}
The complex rank of $D$ homogeneous complex symmetric tensors on $n$ variables is the smallest of the typical ranks of $D$ homogeneous real symmetric tensors of $n$ variables. The maximal rank is at most twice the minimal typical rank.
\end{theorem}

\begin{corollary} If $D$ is odd and $(n,D)$ is not in the list of exceptional cases of Theorem~\ref{ah} then

 \[\kappa(\hat{V}(n+1,D)) \leq \frac{2}{n+1}m(n+1,D) \]
 
\end{corollary}

\begin{proof}
Any convex combination of points in the Veronese variety $x=\sum_{i=1}^a \lambda_i v(x_i)$ can be written as $x=\sum v(((\lambda_i)^{1/D}x_i)$. In turn consider the tensor $X=\sum_{i=1}^a ((\lambda_i)^{1/D}x_i)^{\otimes D}$. By the previous theorem this can be written as a linear combination  By the theorem there exists $k$, with $k \leq \frac{2}{n+1}m(n+1,D)$, and $\{y_1,y_2 \ldots y_k\}$, such that $X=\sum_{i=1}^k (y_i)^{\otimes D}$, and again we might write $x=\sum_{i=1}^k \frac{1}{n} v(n^{1/D} y_i)$.
\end{proof}

\section{Conclusion}

\subsection{Comparison to the standard approach to $t$-nets}

Denote by $N_M (t, m, D)$ the smallest number such that for any Borel measure in $\mu$ in $\mathbb R^m$, there exist a  $t$-net for $(\mu, \poly_+(m,D))$ supported in $M \subset \R^m$. The case $M=\R^m$ corresponds to weak $t$-nets, with a slight abuse of notation, the case where $M$ is the support of $\mu$ corresponds to strong epsilon nets, let us simply denote it by 
$N_{\supp} (t, m, D)$. 

The direct application of the Veronese map produces the following bound on the size of a strong $t$-net for any Borel measure in $\mathbb R^n$
\[
N_{\supp}(t, n, D) \le N_{\supp} \left(t, \binom{n+D}{n}-1, 1\right).
\] 

Using the Veronese map together with the Carath\'eodory theorem in the image, we obtain
\[
N_{\supp}(t, n, D) \le \binom{n+D}{n} \cdot N_{\R^m}\left(t, \binom{n+D}{n}-1, 1\right) 
\]
and
\[
N_{\R^n}(t, n, D) \le \kappa(V(n,D)) \cdot N_{\R^m}\left(t, \binom{n+D}{n}-1, 1\right),
\]
depending on whether we express the points of the $t$-net in the image as convex combinations of points in the support of $v_\#(\mu)$ or points in $V(n,D)$. 

For comparison, the classical results (see \cite{pach2011combinatorial}) imply the following:
\[
N_{\supp}(t, n, D)\leq \frac{\binom{n+D}{n}+1}{t} \log \frac{\binom{n+D}{n}}{t}.
\]

In the regime $t \sim 1$ that our results assume we see that these known bounds are worse than ours. But notice that the classical proof also imply that a random set of $\frac{\binom{n+D}{n}+2}{t} \log \frac{\binom{n+D}{n}}{t}$ points (independently distributed using measure $\mu$) is a strong $t$-net with probability close to $1$, and by making the set slightly larger the probability of failure becomes exponentially small.  It would be interesting to understand if the logarithmic factor is necessary in this regime $t \sim 1$ for this probabilistic statement and an arbitrary measure. 

Finally we recall a lower bound, shown in \cite{kupavskii2016new}, for any $t>0$:

\[
N_{\supp}(t, 2n+2, 1)\geq \frac{n}{9t}\log \frac{1}{t}
\]

\subsection{A question about polynomials}

An interesting question that arises along the above arguments is the following: 

\begin{question}
What is the smallest $N(n,D)$ such that for any set $X$ of $N(n,D)$ points in $\mathbb R^n$ there exists a non negative polynomial of degree at most $D$ that vanishes on every point of $X$?
\end{question}

\bibliographystyle{plain} 
\bibliography{biblio} 

 \end{document}